\documentclass[12pt, reqno]{amsart}
\usepackage{amsmath, amsthm, amscd, amsfonts, amssymb, graphicx, color}
\usepackage[bookmarksnumbered, colorlinks, plainpages]{hyperref}
\hypersetup{colorlinks=true,linkcolor=red, anchorcolor=green, citecolor=cyan, urlcolor=red, filecolor=magenta, pdftoolbar=true}

\newtheorem{theorem}{Theorem}[section]
\newtheorem{lemma}[theorem]{Lemma}

\theoremstyle{definition}
\newtheorem{definition}[theorem]{Definition}
\newtheorem{example}[theorem]{Example}

\theoremstyle{remark}
\newtheorem{remark}[theorem]{Remark}
\numberwithin{equation}{section}
\newcommand{\spt}{{\rm supp}}
\allowdisplaybreaks
\begin{document}
	
	\setcounter{page}{1}
	\title[Hypercyclic Sequences of weighted translations on hypergroups]{Hypercyclic Sequences of weighted translations on hypergroups}
	\author[V. Kumar and S.M. Tabatabaie]{Vishvesh Kumar$^{1,*}$ \MakeLowercase{and} Seyyed Mohammad Tabatabaie$^{2}$}
	
	\address{$^1$Department of Mathematics: Analysis, Logic and Discrete Mathematics
		\endgraf
		Ghent University, Ghent 9000, Belgium}
	\email{\textcolor[rgb]{0.00,0.00,0.84}{vishveshmishra@gmail.com}}
	
	\address{$^{2}$Department of Mathematics, University of Qom, Qom, Iran.}
	\email{\textcolor[rgb]{0.00,0.00,0.84}{sm.tabatabaie@qom.ac.ir}}


	
	\keywords{locally compact hypergroup, hypercyclic operator, hereditary hypercyclic operator, topologically transitivite operator, aperiodic sequence, Orlicz space.\\
		$^*$Corresponding Author}

	\begin{abstract}
		In this paper we characterize hypercyclic sequences of weighted translation operators on an Orlicz space in the context of locally compact hypergroups.    
	\end{abstract} \maketitle
	
	\section{Introduction and preliminaries}
	Linear dynamical properties of bounded operators have been investigated intensely during the last decades; see \cite{gro} as a monograph. Specially, H.N. Salas in \cite{Salas} characterized hypercyclicity of bilateral weighted shifts on $\ell^2(\mathbb{Z})$ by some conditions on the weight function.  
	Then, study of linear dynamic of weighted  translations on a Lebesgue space $L^p(G)$  was begun by \cite{chen11,cc09,che1}, where $G$ is a locally compact group and $1\leq p<\infty$.  Afterwards, the theme on $L^p(G)$ appeared in many articles.  For instance,  the existence of hypercyclic weighted translations on $L^p(G)$ was studied in \cite{kuchen17}. See also \cite{che2, chta3} for the hypergroup case and vector-valued version. In \cite{aa17, cd18} some linear dynamical properties of weighted translation operators on  Orlicz spaces in the context of locally compact groups  have been studied. 
	Orlicz spaces are a generalization of the usual Lebesgue spaces which have been thoroughly investigated over the last decades. Recently, Chen et al. in \cite{che4} 
	gave a characterization of  topologically transitive translation operators on a  weighted Orlicz space $L_w^\Phi(G)$, where $\Phi$ is a Young function, $w$ is a weight and $G$ is a second countable locally compact group. Recently, Orlicz spaces on locally compact hypergroups have been investigated by V. Kumar et al. in \cite{Kumar,KR,Kumar1}. In this paper, we extend many results in recent papers for the context of hypergroups. For this, by fixing a sequence in a hypergroup $K$ and a weight function, we introduce a sequence of bounded linear operators on the Orlicz space $L^\Phi(K)$ and state some necessary conditions for this sequence to be densely hypercyclic. Among other things, the concept of an aperiodic sequence in a hypergroup (see Definition \ref{ape}) plays a key role in the proofs. The main idea for initiating this concept is an equivalent condition given in \cite{che1} for the group case. In sequel, we improve our results for the special case that the sequence of weighted translation operators corresponds to a sequence in the center of hypergroup. Finally, we give an equivalent condition for a single weighted translation operator $T_{z,w}$ to be hereditary hypercyclic on the Orlicz space $L^\Phi(K)$. For convenience of readers, here we write some preliminaries.
	\subsection{Locally Compact Hypergroups} In this subsection we recall definition and some basic properties of hypergroups. For more details we refer to monographs \cite{jew} and \cite{blm}.
	Let $X$ be a locally compact Hausdorff space. We denote by $\mathcal{M}(X)$ the space of all Radon complex measures on $X$, and by $C_c(X)$ the set of all continuous compactly supported complex-valued functions on $X$.  The point mass measure at $x\in X$ and the support of any measure $\mu\in \mathcal{M}(X)$ are denoted by $\delta_x$ and ${\rm supp}(\mu)$ respectively.  For each $A\subseteq X$, $\chi_A$ denotes the characteristic function of $A$.
	\begin{definition}
		Suppose that $K$ is a non-empty locally compact Hausdorff space, $(\mu,\nu)\mapsto \mu\ast\nu$ is a bilinear positive-continuous mapping from $\mathcal{M}(K)\times \mathcal{M}(K)$ into $\mathcal{M}(K)$ (called {\it convolution}), and $x\mapsto x^-$ is an involutive homeomorphism on $K$ (called {\it involution}) with the following properties:
		\begin{enumerate}
			\item[(i) ]$(\mathcal{M}(K),+,\ast)$ is an (associative) algebra;
			\item[(ii) ]for each $x,y\in K$, $\delta_x\ast\delta_y$ is a compact supported probability measure;
			\item[(iii) ]the mapping $(x,y)\mapsto\text{supp}(\delta_x\ast\delta_y)$ from $K\times K$ into $\textbf{C}(K)$ is continuous, where  $\textbf{C}(K)$ is the set of all non-empty compact subsets of $K$ equipped with Michael topology;
			\item[(iv) ]there exists a (necessarily unique) element $e\in K$ (called \emph{identity}) such that $\delta_x\ast \delta_e=\delta_e\ast\delta_x=\delta_x$ for all $x\in K$;
			\item[(v) ]for each $x,y\in K$, $e\in \text{supp}(\delta_x\ast\delta_y)$ if and only if $x=y^-$;
			\item[(vi) ]for each $x,y\in K$, $(\delta_x\ast\delta_y)^-=\delta_{y^-}\ast\delta_{x^-}$.
		\end{enumerate}
		Then, $K\equiv(K,\ast,^-,e)$ is called a \emph{locally compact hypergroup} (or simply a \emph{hypergroup}).
	\end{definition}
	A non-zero non-negative Radon measure $m$ on a hypergroup $K$ is called a \emph{right Haar measure} if for each $x\in K$, $m\ast\delta_x=m$. Throughout $K$ is a locally compact  hypergroup and $m$ is a  right Haar measure on $K$. For each Borel measurable function $f,g:K\rightarrow\mathbb{C}$  and $x,y\in K$ we define the right translation of function $f$ at $x \in K$ by an element $y \in K$ by
	$$f_x(y)=f^y(x)=f(x*y):=\int_K f\,d(\delta_x * \delta_y),$$
	whenever  this integral exists. A locally compact group $G$ equipped with
	$$\mu\ast\nu\mapsto\int_K\int_K\delta_{xy}\,d\mu(x)d\nu(y)\qquad (\mu,\nu\in \mathcal{M}(K))$$
	as convolution, and $x\mapsto x^{-1}$ from $G$ onto $G$ as involution is a hypergroup. In this case trivially we have $f(x\ast y)=f(xy)$. Although any locally compact group is a hypergroup, in general there is no action between elements of a hypergroup; See \cite{blm} for several classes of hypergroups. If $\mu\in \mathcal M(K)$ and $f$ is a Borel measurable function on $K$, the convolutions $f*\mu$ is defined by:
	$$(f*\mu)(x) = \int_K f(x*y^-)\ d\mu(y),\quad(x\in K).$$
	In particular, $(f*\delta_{y^-})(x)=f^y(x).$ 
	
	For any $A,B\subseteq K$ we define
	$$A\ast B:=\bigcup_{x\in A,\,\, y\in B}\text{supp}(\delta_x\ast\delta_y).$$
	
	For each $x\in K$ we denote $\{x\}\ast A$ and $A\ast\{x\}$ simply by $x*A$ and $A*x$. Also, for each $n\in\mathbb N$ we put
	$$\{x\}^n:=\{x\}*\cdots*\{x\} \quad(\text{$n$ times}).$$
	\begin{definition}
		Let $K$ be a hypergroup. The {\it center} of $K$ is defined by
		$$Z(K):=\{x\in K: \delta_x\ast\delta_{x^-}=\delta_{x^-}\ast\delta_{x}=\delta_{e}\}.$$
	\end{definition} 
	Center of a hypergrup $K$, as the maximal subgroup of $K$, was introduced and studied in \cite{dun} and \cite{jew} (see also \cite{ross}). For each $x\in Z(K)$ and $y\in K$, $\text{supp}(\delta_x\ast\delta_y)$ and $\text{supp}(\delta_y\ast\delta_x)$ are singleton, and we denote the single element of them 
	by $xy$ and $yx$, respectively. For each $z \in Z(K)$ and $n \in \mathbb{N}$, setting 
	$$\delta_z^n:=\smash[b]{\underbrace{\delta_z*\delta_z*\cdot \cdots*\delta_z}_{n \text{ times}}},$$ 
	we have $\delta_z^n=\delta_{z^n}.$ 
	\subsection{Basics of Orlicz spaces} Here, we present some definitions and facts related to Orlicz spaces in the context of hypergroups. We refer to  monographs  \cite{ Rao, rao} and articles \cite{Kumar, Kumar1, KR} for more details. 
	
	A non-zero convex function $\Phi : [0, \infty) \rightarrow [0, \infty]$ is called a {\it Young function} if  $\Phi(0)=0$ and $\lim_{x \rightarrow \infty} \Phi(x)=\infty.$ The {\it complimentary function} $\Psi$ of a given Young function $\Phi$  is defined by  
	$$ \Psi(y):= \mbox{sup}\{xy- \Phi(x): x \geq 0\}, \quad   (y \geq 0),$$ 
	which is also a Young function. In this case,  $(\Phi, \Psi)$ is called a {\it complementary pair}. In this paper, we assume that $(\Phi, \Psi)$ is a complementary pair. 
	For a  locally compact  hypergroup $K$ with a right Haar measure $m$, let $L^{\Phi}(K)$
	denotes the set of all Borel measurable
	functions $f: K \rightarrow \mathbb{C}$  such that
	$$\int_K \Phi(\alpha|f(x)|) \,dm(x)<\infty,$$ 
	for some $\alpha>0$.
	For each  function $f\in L^\Phi(K)$ we put
	$$\|f\|_\Phi:=\sup_{v\in\Omega_\Psi}\int_K |fv|\,dm,$$
	where $\Omega_\Psi$ denotes the set of all complex-valued Borel measurable functions $v$ on $K$ satisfying $\int_K \Psi(|v(x)|)\,dm(x)\leq 1$. 
	Then, since $m$ is a Haar measure, by \cite[Chapter III, Proposition 11]{rao}, $(L^\Phi(K),\|\cdot\|_\Phi)$ is a Banach space called an \emph{Orlicz space}.
	Another equivalent norm on $L^{\Phi}(G)$ is defined by
	$$N_\Phi(f):=\inf \\ \left\{ k>0: \int_K \Phi\left(\frac{|f|}{k} \right)dm \leq 1 \right\},$$
	for all $f\in L^\Phi(K)$, and called the \emph{Luxemburg norm}. In fact, we have 
	$$N_\Phi(f)\leq \|f\|_\Phi\leq 2N_\Phi(f)$$
	for all $f\in L^\Phi(K)$. 
	A Young function $\Phi$ is said to be
	\emph{$\Delta_2$-regular} and we write $\Phi\in \Delta_2$, if there are constants $k>0$ and $t_0\geq 0$ such that
	$\Phi(2t)\leq k\Phi(t)$ for each $t\geq t_0$. If $\Phi$ is
	$\Delta_2$-regular, then the space $C_c(K)$ is dense in $L^{\Phi}(K)$.
	For each $1 \leq p < \infty$, the function $\Phi_p$ defined by  $\Phi_p(x):=\frac{x^p}{p}$ is a Young function and the Orlicz space $L^{\Phi_p}(K)$ is same as the Lebesgue space $L^p(K)$. The complementary Young function of $\Phi_p$ is $\Phi_q$,  where $q=\frac{p}{p-1}$.  Other examples of Young functions includes $e^x-x-1,\, \cosh x-1$ and  $x^p \ln(x)$; see \cite{Rao} for more examples.
	
	If $(X,\mu)$ is a  probability measure space and $f$ is a real-valued measurable function on $X$ such that $\int_X f\,d\mu$ and $\int_X \Phi(f)\,d\mu$ exist, then by the Jensen's inequality \cite[Proposition 5, p. 62]{Rao} we have
	\begin{equation}\label{jen}
	\Phi\left(\int_X f(x)\,d\mu(x)\right)\leq \int_X \Phi(f(x))\,d\mu(x).
	\end{equation}
	
	For any $f \in L^{\Phi}(K),$ it is shown in \cite[Corollary 3.2]{Kumar} that $N_{\Phi}(f^z)\leq N_{\Phi}(f)$ for $z \in K.$ But if we take $z \in Z(K)$ we have following strong result which says norm $N_\Phi(\cdot)$ is invariant under translation by an element of the center of a hypergroup.
	
	\begin{lemma} \label{1.1} For $z \in Z(K)$ and $f \in L^\Phi(K),$ we have $N_\Phi(f^z)=N_\Phi(f).$
	\end{lemma}
	\begin{proof}
		Let $f \in L^\Phi(K)$ and $z \in Z(K).$ Then 
		\begin{eqnarray} 
		N_\Phi(f^z)= \inf\left\lbrace k>0: \int_K \Phi \left( \frac{|f^z(x)|}{k} \right)\, dm(x) \leq 1 \right\rbrace
		\end{eqnarray}
		
		By setting $g(x):= \Phi\left(\frac{|f^z(x)|}{k} \right) \geq 0$ for all $x\in K$, and using the property \cite[3.3F]{jew} for right Haar measure $m,$ we get  
		\begin{eqnarray*}
			\int_K \Phi \left( \frac{|f^z(x)|}{k} \right)\, dm(x) = \int_K g(x)\, dm(x) = \int_K g^{z^-}(x)\, dm(x)= \int_K g(x*z^-)\, dm(x).
		\end{eqnarray*}
		Since $z \in Z(K),$ $\spt(\delta_x * \delta_{z^-})=\{xz^- \}$ is singleton and hence we get 
		\begin{eqnarray*}
			\int_K \Phi \left( \frac{|f^z(x)|}{k} \right)\, dm(x) = \int_K g(xz^-)\, dm(x)= \int_K \Phi \left( \frac{|f^z(xz^-)|}{k} \right)\, dm(x).
		\end{eqnarray*}
		Since $z \in Z(K)$ implies that  $\delta_{z^-}*\delta_z= \delta_e$ we have
		\begin{eqnarray*}
			\int_K \Phi \left( \frac{|f^z(x)|}{k} \right)\, dm(x) = \int_K \Phi \left( \frac{|f_x(z^-*z)|}{k} \right)\, dm(x)= \int_K \Phi \left( \frac{|f(x)|}{k} \right) \, dm(x).
		\end{eqnarray*}
		Therefore,
		\begin{eqnarray*}
			N_\Phi(f^z) &=& \inf\left\lbrace k>0: \int_K \Phi \left( \frac{|f^z(x)|}{k} \right)\, dm(x) \leq 1 \right\rbrace\\ &=& \inf\left\lbrace k>0: \int_K \Phi \left( \frac{|f(x)|}{k} \right)\, dm(x) \leq 1 \right\rbrace = N_\Phi(f).
	\end{eqnarray*} \end{proof}
	
	It is known that $L^\Phi(K)$ is Banach module over $M(K)$ \cite[Lemma 3.6]{Kumar}. Therefore, for any $f\in L^\Phi(K)$ and $\mu \in M(K),$ we have $f*\mu \in L^\Phi(K).$ In particular, $f*\delta_{z^n} \in L^\Phi(K)$ for all $f \in L^\Phi(K),\, z\in Z(K) $ and $n \in \mathbb{N}.$
	\section{Main Results}
	In beginning of this section, we recall some linear dynamic concepts of operators which we need in this paper.  
	\begin{definition}
		Let $\mathcal X$ be a Banach space. A sequence $(T_n)_{n \in \mathbb{Z}_+}$ of bounded linear operators from $\mathcal X$ into $\mathcal X$ is called \emph{hypercyclic} if there exists an element $x$ in $\mathcal X$ (called \emph{hypercyclic vector}) such that the set $\{T_0(=I_{\mathcal{X}}) x,T_1x,\ldots\}$ is dense in $\mathcal X$. We say that $(T_n)$ is \emph{densely hypercyclic} if the set of all its hypercyclic vectors is dense in $\mathcal X$. A bounded linear operator $T$ on $\mathcal X$ is called \emph{hypercyclic} (\emph{densly hypercyclic}) if the sequence $(T^n)$ is hypercyclic (densely hypercyclic), where $T^n$ is the $n$-th iterate of $T$.
	\end{definition} 
	The set of all bounded linear operators on a Banach space $\mathcal X$ is denoted by $B(\mathcal X)$.
	This is well-known that there is a hypercyclic operator on a Banach space $\mathcal X$ if and only if  the space $\mathcal X$ is separable and infinite dimensional.
	
	\begin{definition}
		Let $\mathcal X$ be a Banach space. A sequence $(T_n)$ in $B(\mathcal X)$ is called \emph{topologically transitive} if for each two non-empty open sets $U, V \subseteq \mathcal X$ there exists $n \in \mathbb{N}$ such that $T_n(U) \cap V \neq \varnothing$. If the condition $T^n(U) \cap V \neq \varnothing$ holds for some $n$ onwards then $(T_n)$ is called \emph{topologically mixing}.  A bounded linear operator $T$ on  $\mathcal X$ is called \emph{topologically transitive} (\emph{topologically mixing}) if the sequence $(T^n)$ is topologically transitive (topologically mixing). 
	\end{definition}
	In this section, $K$ denotes a hypergroup equipped with a right Haar measure $m$, and  $\Phi$ denotes a Young function. We find some necessary and sufficient conditions for a sequence of operators on $L^\Phi(K)$, generated by a weight and an aperiodic sequence in $K$, to be hypercyclic.
	A main tool in the proof of results is the concept aperiodic sequence.
	
	An element $a$ in a locally compact group $G$
	is called {\it aperiodic } (or non-compact) if the closed subgroup of $G$ generated by $a$ is not compact. In \cite[Lemma 2.1]{che1} it is proved that if $G$ is a second countable group, then $a\in G$ is aperiodic if and only if for each compact subset $E$ of $G$, there exits $N>0$ such that $E\cap Ea^{n}=\varnothing$ (and so $E\cap Ea^{-n}=\varnothing$) for all $n\geq N$. This characterization of aperiodic elements of a group leads to introduce the following suitable analogues of the notion of aperiodic element in the setting of hypergroup. We recall the definition of aperiodicity of a center element of $K$ from \cite[Definition 3.3]{che2}.
	
	\begin{definition} An element $z \in Z(K)$ is called {\it aperiodic}  if for each compact subset $E \subseteq K$ with $m(E)>0,$ there exists $N \in \mathbb{N}$ such that $E \cap (E*\{z\}^n)= \varnothing$ for all $n \geq N.$ 
	\end{definition}
	
	The following lemma gives an equivalent condition of aperiodicity of a center element of a hypergroup.
	\begin{lemma} \cite[Lemma 3.4]{che2}. \label{t3.4} An element $z \in Z(K)$ is aperiodic if and only if for each compact subset $E \subset K$ with $m(E)>0,$ there exists $N \in \mathbb{N}$ such that $(E*\{z\}^{rn}) \cap (E*\{z\}^{sn}) = \varnothing$ for $n \neq N$ and $r,s \in \mathbb{Z}$ with $r \neq s,$ where $\{z\}^{-n}= \{z^-\}^n.$ 
	\end{lemma} 
	Now, we give a version of this definition for a sequence of elements of hypergroup.
	\begin{definition}\label{ape}
		Let $K$ be a locally compact hypergroup, and $\eta:=(a_n)_{n\in\mathbb Z}\subseteq K$. Then, $\eta$  is called an {\it aperiodic sequence} in $K$  if
		\begin{enumerate}
			\item $a_0=e$ and for all $n\in\mathbb{N}$, $a_{-n}=a_n^-$;
			\item for every compact subset $E$ of $K$ with $m(E)>0$, there exists $N>0$ such that for any $n\geq N$,  $E\cap(E\ast\{a_{\pm n}\})=\varnothing$.
		\end{enumerate}
	\end{definition}
	By the above Definition \ref{ape} and \cite[Lemma 2.1]{che1}, if $G$ is a second countable group and $a\in G$,  setting $a_n:=a^n$ for all $n\in\mathbb Z$,  $a$ is aperiodic if and only if $\eta:=(a_n)_{n\in\mathbb Z}$ is an aperiodic sequence in $G$.
	Now, we present some examples of aperiodic sequences in some hypergroups. 
	\begin{example}
		Let $0<a\leq\frac{1}{2}$, and $\mathbb{N}_0:=\{0,1,2,\ldots\}$. For each $r,s\in\mathbb{N}_0$, define
		$$\delta_r\ast\delta_s:=\left\{
		\begin{array}{ll}
		\delta_{\text{max}\{r,s\}},& \mbox{ if } r\neq s,\\\\
		\frac{a^r}{1-a}\delta_0+\sum_{k=1}^{r-1}a^{r-k}\delta_k+\frac{1-2a}{1-a}\delta_r, & \mbox{ if } r=s.
		\end{array}
		\right.$$
		
		Then, $(\mathbb{N}_0,\ast)$ is a Hermitian discrete hypergroup.
		This important class of hypergroups was introduced by Dunkl and Ramirez  in \cite{dun2}. Recently, in \cite{KKA,KKAadd} first author with Singh and Ross studied classification results of such classes of hypergroups arising from the discrete semigroups (see also \cite{Voit}) with applications to Ramsey theory \cite{KumarRamsey}. The above convolution shows that if $E$ is a non-trivial compact (and so finite) subset of $\mathbb{N}_0$ and $n\in\mathbb{N}_0\backslash E$ is greater than all elements of $E$, then $E\ast\{n\}=\{n\}$. Now, let $(\beta_n)_{n=1}^\infty$ be an unbounded sequence in $\mathbb{N}_0$. Put $a_n=a_{-n}:=\beta_n$ for all $n\in \mathbb{N}$ and $a_0:=0$. Then, $\eta:=(a_n)_{n\in\mathbb Z}$ is an aperiodic sequence in the hypergroup $(\mathbb{N}_0,\ast)$.
	\end{example}
	\begin{example}\label{exa2}
		Consider the group $SU(2)$ of all unitary transformations of $\Bbb{C}^2$  with determinant 1. Then, the set of all continuous unitary irreducible representations of the group $SU(2)$ can be indexed by $\mathbb{N}_0\equiv\{T^{(0)}, T^{(1)}, T^{(2)}, \ldots\}$ where $T^{n}$ has dimension $n$, and for each $m,n\in\mathbb{N}_0$, the tensor product of $T^{(n)}$ and $T^{(m)}$ is unitarily equivalent to
		$$T^{|m-n|}\oplus T^{|m-n|+2}\oplus\ldots.\oplus T^{(m+n)}.$$
		See \cite[Example 29.13]{hew}. Then, the discrete space $\mathbb{N}_0$ equipped with the convolution
		\begin{equation}\label{con}
		\delta_m\ast\delta_n:=\sum_{k=|m-n|}^{m+n} \!\!\!\!\!\!\!\!\text{o} \,\,\,\,\,\frac{k+1}{(m+1)(n+1)}\delta_k
		\end{equation}
		is a Hermitian discrete hypergroup, where $\sum \!\!\!\text{o}$  denotes that only every second term appears in the sum. This hypergroup is called $SU(2)$--hypergroup  (for more details refer to \cite[1.1.15]{blm}). Then, one can see that every sequence $(a_n)_{n\in\mathcal Z}$ in $\mathbb{N}_0$ with $a_0:=0$ and $a_{-n}=a_n^-(=a_n)$ is aperiodic if and only if  it does not have any constant subsequence. For this, suppose that $(a_n)_{n\in\mathbb Z}$ does not have any constant subsequence, and let $E$ be a compact (and so finite) subset of $\mathbb{N}_0$ with $m(E)>0$. Then, there is some $N>0$ such that for each $n\geq N$,  $a_n\geq 1+2\max E$. So by  \eqref{con}, for each $n\geq N$, we have $E\cap(E\ast\{ a_n\})=\varnothing$, and hence $(a_n)_{n\in\mathbb Z}$ is aperiodic. The converse is trivial. A similar conclusion can be proved for general polynomial hypergroups.
	\end{example}
	In the below definition we define a sequence of operators on an Orlicz space by a fixed weight function and sequence of elements of the given hypergroup.
	\begin{definition}\label{lamd}
		Any bounded continuous function $w:K\rightarrow(0,+\infty)$ is called a {\it weight} on $K$.
		Suppose that $\eta:=(a_n)_{n\in\mathbb Z}$ is a sequence in $K$, and $w$ is a weight function on $K$. For each $x\in K$ and $n\in\mathbb N$, we define a sequence of operators $(\Lambda_n)$ on $L^\Phi(K)$ by
		\begin{equation}\label{lam}
		\Lambda_n f(x):=w(x)\cdot w(x\ast a_{-1})\cdots w(x\ast a_{-n})\cdot f(x\ast a_{-n})\qquad (f\in L^\Phi(K)).
		\end{equation}
	\end{definition}
	The following lemma shows that for any $n\in\mathbb N$, the operator $\Lambda_n:L^\Phi(K)\rightarrow L^\Phi(K)$ is well-defined.
	
	In sequel of this paper we assume that $(\Phi,\Psi)$ is a complimentary pair such that $\Psi$ is increasing.
	\begin{lemma} \label{3.2} Let $w$ be a weight on $K$, and $a\in K$. Then, $wf^a \in L^\Phi(K)$ and $\|wf^a\|_\Phi\leq \|w\|_{\sup}\,\|f\|_\Phi$ for all $f \in L^\Phi(K)$.
	\end{lemma}
	\begin{proof}
		Let $a\in K$. For each $f\in L^\Phi(K)$ we have 
		\begin{align*}
		\|wf^a\|_\Phi&=\sup_{v\in\Omega_\Psi}\int_K |wf^av|\,dm\\
		&\leq \|w\|_{\sup}\,\sup_{v\in\Omega_\Psi}\int_K |f(x\ast a)v(x)|\,dm(x)\\
		&=\|w\|_{\sup}\,\sup_{v\in\Omega_\Psi}\int_K |f(x)v^{a^-}(x)|\,dm(x)\\
		&\leq\|w\|_{\sup}\,\sup_{v\in\Omega_\Psi}\int_K |f(x)v(x)|\,dm(x)\\
		&=\|w\|_{\sup}\,\|f\|_\Phi,
		\end{align*}
		since for each $v\in\Omega_{\Psi}$ and $b\in K$ we have $v^b\in\Omega_\Psi$. To see this fact, first note that for any $v \in \Omega_\Phi$ and $b \in K,$  we have
		\begin{eqnarray*}
			\int_K \Psi(|v^b(x)|)\, dm(x) &=&\int_K \Psi\left( \left|\int_K v(z) \,d(\delta_x*\delta_b)(z)\right| \right) \, dm(x)\\  &\leq& \int_K \Psi \left( \int_K |v(z)| \,d(\delta_x*\delta_b)(z) \right) \, dm(x).\end{eqnarray*}
		By using Jensen's inequality \cite[Proposition 5, p. 62]{Rao} we get,
		\begin{eqnarray*}
			\int_K \Psi(|v^b(x)|)\, dm(x) &\leq & \int_K \left( \int_K \Psi\left(|v|\right)(z)\,\, d(\delta_x*\delta_b)(z) \right) \, dm(x) \\& =& \int_K \Psi\left( |v|\right)^b(x) \,dm(x) = \int_K \Psi\left( |v|\right)(x) \,dm(x).
		\end{eqnarray*}   
		where the last equality follows from \cite[Lemma 3.3 F]{jew}. Therefore, if $v \in \Omega_\Psi,$ i.e., $\int_K \Psi\left( |v|\right)(x) \,dm(x) \leq 1$ then $\int_K \Psi\left( |v^b(x)|\right) \,dm(x) \leq 1$ and hence $v^b \in \Omega_\Psi.$ 
	\end{proof}
	\begin{remark}
		Suppose that $K$ is a locally compact hypergroup, $a\in K$, and $w$ is a weight on $K$. Then, define the bounded linear operator $T_{a,w}$ on the Orlicz space $L^\Phi(K)$ by
		\begin{equation}
		T_{a,w}f(x):=w(x)f(x\ast a^-),\quad(f\in L^\Phi(K)).
		\end{equation}
		In the case that $K$ is a locally compact group, for each $f\in L^\Phi(K)$, $n\in\mathbb N$ and $x\in K$, we have
		$$\Lambda_n f(x)=w(x)\cdot w(xa_{-1})\cdots w(xa_{-n})\cdot f(xa_{-n}).$$
		In particular, if $a\in K$, then  we have  $T_{a,w}^n=\Lambda_n$ for all $n\in\mathbb{N}_0$, where $(\Lambda_n)$ is the sequence defined by \eqref{lam} corresponding to  $\eta:=(a^n)_{n\in\mathbb Z}$. Since in general we do not have any action between elements of a hypergroup, we study the linear dynamic properties of the sequence $(\Lambda_n)$ of operators as in Definition \ref{lamd} for hypergroups.
	\end{remark}
	Although the term \emph{hypercyclic} is used for bounded linear operators on Banach spaces, since the above sequence $(\Lambda_n)$ is determined by a given weight  $w$ and a sequence $\eta$, this term has been used for weights in next definition.
	
	\begin{definition}\label{hyper}
		Let $\eta:=(a_n)_{n\in\mathbb{Z}}$ be a sequence in $K$, and $\Phi$ be a Young function. A weight $w$ on $K$ is called $(\eta,\Phi)$-{\it hypercyclic}  if there is a function $f\in L^\Phi(K)$ (called a {\it hypercyclic vector}) such that the set $\{ f,\Lambda_1 f,\Lambda_2 f,\ldots\}$ is dense in $L^\Phi(K)$. An $(\eta,\Phi)$-hypercyclic weight is called {\it densely} $(\eta,\Phi)$-hypercyclic if the set of its hypercyclic vectors is dense in $L^\Phi(K)$. Also, an $(\eta,\Phi)$-hypercyclic weight is called  {\it positively densely hypercyclic} if for each $g\in L^\Phi_+(K)$ and $\varepsilon>0$, there exist a vector $f\in L^\Phi(K)$ and a large enough natural number $n$ such that $\|\Lambda_nf-g\|_\Phi<\varepsilon$ and $\|f-g\|_\Phi<\varepsilon$.
	\end{definition}
	
	\begin{remark}\label{222}
		Let $\eta:=(a_n)_{n\in\mathbb{Z}}$ be a sequence in a hypergroup $K$,  $\Phi$ be a Young function and $w$ be a weight function on $K$. Then, $w$ is $(\eta,\Phi)$- hypercyclic if and only if the sequence $(\Lambda_n)$ given by \eqref{lam} is a hypercyclic sequence of operators on $L^\Phi(K)$.  In particular, if $K$ is a locally compact group and $a$ is an  element of $K$, then by the above remark and definition, the operator  $T_{a,w}:L^\Phi(K)\rightarrow L^\Phi(K)$ is hypercyclic if and only if $w$ is $(\eta,\Phi)$-hypercyclic, where $\eta:=(a^n)_{n\in\mathbb{Z}}$. 
	\end{remark}
	
	In the following result we  give a necessary condition for a weight to be positively densely $(\eta,\Phi)$-hypercyclic whenever the condition $L^\Phi(K)\subseteq L^1(K)$ satisfies. In Remark \ref{ddd} we give some explanations about this inclusion.
	\begin{theorem}\label{33}
		Let $K$ be a locally compact hypergroup and let $w$ be a weight on $K$.
		Suppose that $\eta:=(a_n)_{n\in\mathbb Z}$ is an aperiodic sequence in $K$ and $\Phi$ is a Young function such that $L^\Phi(K)\subseteq L^1(K)$.
		If $w$ is positively densely $(\eta,\Phi)$-hypercyclic, then for every compact subset $E\subseteq K$ with $m(E)>0$, there exist a sequence $(E_k)_{k=1}^\infty$ of Borel subsets of  $E$, and  a strictly increasing  sequence $(n_k)_{k=1}^\infty\subseteq\mathbb{N}$
		such that $\lim_{k\rightarrow\infty}m(E_k)=m(E)$ and
		$\lim_{k\rightarrow\infty}\|w_{n_k}|_{E_k}\|_\infty=0,$
		where for all $n\in\mathbb{N}$ and $x\in K$, 
		$v_n(x):=w(x)\cdot w(x\ast a^{-}_{1})\cdots w(x\ast a^{-}_{n}),$
		and  $w_n:=((\chi_E)^{a^{-}_{n}} v_n)^{a_n}$. 
	\end{theorem}
	\begin{proof}
		Let $w$ be a positively densely $(\eta,\Phi)$-hypercyclic weight, and $E$ be a compact subset of $K$ with $m(E)>0$.  By Definition  \ref{ape} and Definition \ref{hyper}, there exist a sequence  $(f_k)_{k=1}^\infty\subseteq L^\Phi_+(K)$ and a strictly increasing sequence $(n_k)_{k=1}^\infty\subseteq\mathbb{N}$ such that for each $k=1,2,\ldots$, $E\cap(E\ast\{a_{n_k}\})=\varnothing$ and
		$$\|f_k-\chi_E\|_\Phi<\frac{1}{4^{k}},\quad \|\Lambda_{n_k}f_k-\chi_E\|_\Phi<\frac{1}{4^k}.$$
		By the Closed Graph Theorem, from the inclusion $L^\Phi(K)\subseteq L^1(K)$ one can conclude that there exists a constant $M>0$ such that for each $f\in L^\Phi(K)$,
		\begin{equation}
		M\,\|f\|_1\leq \|f\|_\Phi.
		\end{equation}
		Set $A_k:=\{x\in E: |f_k(x)-1|\geq \frac{1}{2^k}\}$. Then, 
		\begin{align*}
		\frac{1}{M 4^k}>\frac{1}{M}\|f_k-\chi_E\|_\Phi&\geq \|f_k-\chi_E\|_1=\int_K|f_k(x)-\chi_E(x)|\,dm(x)\\
		&\geq \int_{E}|f_k(x)-1|\,dm(x)\geq \frac{1}{2^k} m(A_k),
		\end{align*}
		and so, $m(A_k)<\frac{1}{M\,2^k}$. 
		Put $B_k:=\{x\in E:((\chi_E)^{a^-_{n_k}} v_{n_k})^{a_{n_k}}(x)f_k(x)\geq \frac{1}{2^k}\}$. Then,
		\begin{align*}
		\frac{1}{M 4^k}&>\frac{1}{M}\|\Lambda_{n_k}f_k-\chi_E\|_\Phi
		\geq \|\Lambda_{n_k}f_k-\chi_E\|_1=\int_K |(\Lambda_{n_k}f_k-\chi_E)(x)|\,dm(x)\\
		&\geq \int_K |(\Lambda_{n_k}f_k-\chi_E)^{a_{n_k}}(x)|\,dm(x)
		=\int_K|\Lambda_{n_k}f_k(x\ast a_{{n_k}})-\chi_E(x\ast a_{n_k})|\,dm(x)\\
		&\geq \int_E|\Lambda_{n_k}f_k(x\ast a_{{n_k}})-\chi_E(x\ast a_{n_k})|\,dm(x)
		=\int_E |\Lambda_{n_k}f_k(x\ast a_{{n_k}})|\,dm(x)\\
		&=\int_K \chi_E(x) \cdot \Lambda_{n_k}f_k(x\ast a_{{n_k}})\,dm(x)
		=\int_K \chi_E(x\ast a_{-{n_k}}) \cdot \Lambda_{n_k}f_k(x)\,dm(x).
		\end{align*}
		Now, by the definition of operator $\Lambda_{n_k}$ we have
		\begin{align*}
		\frac{1}{M 4^k} & \geq \int_K \chi_E(x\ast a_{-{n_k}}) v_{n_k}(x) f_k(x\ast a_{-{n_k}}) \,dm(x)\\
		&=\int_K ((\chi_E)^{a_{-{n_k}}} v_{n_k})(x)f_k(x\ast a_{-{n_k}}) \,dm(x)
		=\int_K ((\chi_E)^{a_{-{n_k}}} v_{n_k})(x\ast a_{n_k})f_k(x) \,dm(x)\\
		&=\int_K ((\chi_E)^{a_{-{n_k}}} v_{n_k})^{a_{n_k}}(x)f_k(x) \,dm(x)
		\geq \int_{B_k} ((\chi_E)^{a_{-{n_k}}} v_{n_k})^{a_{n_k}}(x)f_k(x) \,dm(x)\\
		&\geq \frac{1}{2^k} m(B_k),
		\end{align*}
		where $v_n(x):=w(x)\cdot w(x\ast a_{-1})\cdots w(x\ast a_{-n})$. So, $m(B_k)<\frac{1}{M 2^k}$. Setting $E_k:=E\backslash(A_k\cup B_k)$ we have $\lim_{k\rightarrow\infty}m(E_k)=m(E)$.
		
		Also, let $w_n:=((\chi_E)^{a^-_{n}} v_n)^{a_n}$. Then, for each $x\in E_k$, one has
		$$w_{n_k}(x)<\frac{1}{2^k f_k(x)}<\frac{1}{2^k(1-\frac{1}{2^k})},$$
		and so, $\lim_{k\rightarrow\infty}\|w_{n_k}|_{E_k}\|_\infty=0$.
	\end{proof}
	\begin{definition}
		A weight is called
		(positively) {\it chaotic} if it is (positively) densely hypercyclic and the set of its periodic element is (positively) dense.
	\end{definition}

	\begin{definition}\label{pe}
		Let $\eta:=(a_n)_{n\in\mathbb Z}$ be a sequence of elements in $K$, $\Phi$ be a Young function, and $w$ be a weight on $K$.
		A function $f\in L^\Phi(K)$ is called $(\eta,w)$-{\it periodic} if there exists a number $n\in\mathbb N$ such that for all $r\geq 1$, $\Lambda_{rn}f=f$. The set of all $(\eta,w)$-periodic elements  is denoted by $P_{\eta,w}$.
	\end{definition}
	\begin{definition}\label{ape2}
		Let $K$ be a locally compact hypergroup, and $\eta:=(a_n)_{n\in\mathbb Z}\subseteq K$. Then, $\eta$ is called {\it strongly aperiodic} if
		\begin{enumerate}
			\item $a_0=e$, and for each $n\in\mathbb{N}$, $a^-_n=a_{-n}$;
			\item for all compact subset $E$ of $K$ with $m(E)>0$, there exists $N>0$ such that for any $n\geq N$ and all distinct $r,s\in \mathbb{Z}$,
			$$(E\ast\{a_{rn}\})\cap (E\ast\{a_{sn}\})=\varnothing.$$
		\end{enumerate}
	\end{definition}
	\begin{remark}
		If $G$ is a locally compact group, then for each set $E\subseteq G$, $a\in G$ satisfies $E\cap Ea^{\pm n}=\varnothing$ from some $n$ onward if and only if
		$Ea^{rn}\cap Ea^{sn}=\varnothing$ holds for any distinct integers $r$ and $s$. However, this is not the case for general hypergroups. Therefore, the condition (2)  in the above definition is a suitable replacement for the latter condition.
	\end{remark}
	\begin{theorem}\label{22}
		Let $K$ be a locally compact hypergroup, $w$ be a weight on $K$, and $\Phi$ be a Young function such that $L^\Phi(K)\subseteq L^1(K)$. Let $\eta:=(a_n)_{n\in\mathbb Z}$ be a strongly aperiodic sequence in $K$. If $P_{\eta,w}$ is dense in $L^\Phi_+(K)$, then for every compact subset $E\subseteq K$ with $m(E)>0$, there exist a sequence $(E_k)_{k=1}^\infty$ of subsets of $E$, and a strictly increasing sequence $(n_k)_{k=1}^\infty\subseteq\mathbb{N}$
		such that  $\lim_{k\rightarrow\infty}m(E_k)=m(E)$ and
		$$\lim_{k\rightarrow \infty}\left(\sum_{r=1}^\infty\int_{E_k}(\chi_{E}^{a_{-sn_k}}v_{sn_k})^{a_{sn_k}}\,dm+\sum_{s=1}^\infty\int_{E_k}v_{sn_k}(x)^{-1}\,dm(x)\right)=0,$$
		where for all $n\in\mathbb{N}$ and $x\in K$,
		$v_n(x):=w(x)\cdot w(x\ast a^-_{1})\cdots w(x\ast a^-_{n}).$
	\end{theorem}
	\begin{proof}
		Assume that $P_{\eta,w}$ is dense in $L_+^\Phi(K)$. Let $E\subseteq K$ be compact and $m(E)>0$. Since $\eta$ is strongly aperiodic, by Definition \ref{ape2} there is a constant $N>0$ such that for each $n\geq N$ and distinct $r,s\in \mathbb{Z}$,  
		$$(E\ast\{a_{rn}\})\cap (E\ast\{a_{sn}\})=\varnothing.$$
		
		Since $\chi_K\in L_+^\Phi(K)$, there is a sequence $(f_k)_{k=1}^\infty\subseteq P_{\eta,w}$ such that for each $k\in\mathbb N$,  $\|f_k-\chi_E\|_\Phi<\frac{1}{4^k}$. By Definition \ref{pe}, for each $k=1,2,\ldots$ there exists  $n_k\in\mathbb N$ such that for all $r\geq 1$, $\Lambda_{rn_k}f_k=f_k$. Clearly, we can suppose that $n_{k+1}>n_k\geq N$ for all $k\in\mathbb N$.
		Since $L^\Phi(K)\subseteq L^1(K)$, there is a constant $M>0$ such that for each $f\in L^\Phi(K)$,
		$$M \|f\|_1\leq \|f\|_\Phi.$$
		
		For each $k=1,2,\ldots$, we put $A_k:=\{x\in E:|f_k(x)-1|\geq \frac{1}{2^k}\}$. So, same as the proof of Theorem \ref{33} we have  $m(A_k)<\frac{1}{M 2^{k}}$.
		Also, setting $B_k:=\{x\in E: |v_{sn_k}(x)f_k(x\ast a_{-sn_k})|\geq\frac{1}{2^k}\}$ we have  $m(B_k)<\frac{1}{M 2^k}$. Indeed,
		\begin{align*}
		\frac{1}{4^k}&\geq \|f_k-\chi_E\|_\Phi\geq M\|f_k-\chi_E\|_1=M\|\Lambda_{sn_k}f_k-\chi_E\|_1\\
		&=M\int_K |\Lambda_{sn_k}f_k(x)-\chi_E(x)|\,dm(x)\geq M\int_{B_k}|v_{sn_k}(x)f_k(x\ast a_{-sn_k})-1|\,dm(x)\\
		&\geq \frac{M}{2^k}m(B_k).
		\end{align*}
		
		For each $k\in\mathbb N$ we set $E_k:=E\backslash(A_k\cup B_k)$. Then, by the above inequalities, $\lim_{k\rightarrow\infty}m(E_k)=m(E)$. Moreover,  for each $k=1,2,\ldots$, 
		\begin{align*}
		\int_{K\backslash E}f_k(x)\,dm(x)&\geq\sum_{r=1}^\infty\left(\int_{E\ast\{a_{rn_k}\}}f_k(x)\,dm(x)\right)+\sum_{s=1}^\infty\left(\int_{E\ast\{a_{-sn_k}\}} f_k(x)\,dm(x)\right)\\
		&=\sum_{r=1}^\infty\left(\int_K \chi_{E\ast\{a_{rn_k}\}}(x)f_k(x)\,dm(x)\right)\\
		&\hspace{4cm}+\sum_{s=1}^\infty\left(\int_K \chi_{E\ast\{a_{-sn_k}\}}(x)f_k(x)\,dm(x)\right)\\
		&\geq \sum_{r=1}^\infty\left(\int_K\chi_{E}(x\ast a_{-sn_k})f_k(x)\,dm(x)\right)\\
		&\hspace{4cm}+\sum_{s=1}^\infty\left(\int_K \chi_{E}(x\ast a_{sn_k})f_k(x)\,dm(x)\right)\\
		&=\sum_{r=1}^\infty\left(\int_K\chi_{E}(x\ast a_{-sn_k})\Lambda_{sn_k}f_k(x)\,dm(x)\right)\\
		&\hspace{4cm}+\sum_{s=1}^\infty\left(\int_K \chi_{E}(x)f_k(x\ast a_{-sn_k})\,dm(x)\right)
		\end{align*}
		Now by using the definition of operator $\Lambda_{sn_k}$ we get
		\begin{align*}
		\int_{K\backslash E}f_k(x)\,dm(x)&\geq\sum_{r=1}^\infty\left(\int_K\chi_{E}(x\ast a_{-sn_k})v_{sn_k}(x)f_k(x\ast a_{-sn_k})\,dm(x)\right)\\
		&\hspace{4cm}+\sum_{s=1}^\infty\left(\int_{E_k}f_k(x\ast a_{-sn_k})\,dm(x)\right)
		\end{align*}
		\begin{align*}
		&=\sum_{r=1}^\infty\left(\int_K((\chi_{E})^{a_{-sn_k}}v_{sn_k})^{a_{sn_k}}(x)f_k(x)\,dm(x)\right)\\
		&\hspace{4cm}+\sum_{s=1}^\infty\left(\int_{E_k}v_{sn_k}(x)^{-1}v_{sn_k}(x)f_k(x\ast a_{-sn_k})\,dm(x)\right)\\
		&\geq\sum_{r=1}^\infty\left(\int_{E_k}(\chi_{E})^{a_{-sn_k}}v_{sn_k})^{a_{sn_k}}(x)f_k(x)\,dm(x)\right)\\
		&\hspace{4cm}+\sum_{s=1}^\infty\left(\int_{E_k}v_{sn_k}(x)^{-1}v_{sn_k}(x)f_k(x\ast a_{-sn_k})\,dm(x)\right)\\
		&\geq(1-\frac{1}{2^k})\sum_{r=1}^\infty\left(\int_{E_k}((\chi_{E})^{a_{-sn_k}}v_{sn_k})^{a_{sn_k}}\,dm\right)\\
		&\hspace{4cm}+(1-\frac{1}{2^k})\sum_{s=1}^\infty\left(\int_{E_k}v_{sn_k}(x)^{-1}\,dm(x)\right)
		\end{align*}
		thanks to the right invariance of the Haar measure $m$ and Definition \ref{ape2}. Therefore,
		\begin{align*}
		\frac{1}{4^{k}}&>\|f_k-\chi_E\|_\Phi\geq M\|f_k-\chi_E\|_1 \geq M\int_{K\backslash E} |f_k(x)|\,dm(x)\\
		&\geq M(1-\frac{1}{2^k})\sum_{r=1}^\infty\left(\int_{E_k}(\chi_{E}^{a_{-sn_k}}v_{sn_k})^{a_{sn_k}}\,dm\right)\\
		&\hspace{4cm}+M(1-\frac{1}{2^k})\sum_{s=1}^\infty\left(\int_{E_k}v_{sn_k}(x)^{-1}\,dm(x)\right),
		\end{align*}
		and the proof of the theorem is completed.
	\end{proof}
	\begin{remark}\label{ddd}
		Note that by \cite[Theorem 2]{hud}, if $K$ is a locally compact group and $\Phi$ is a Young function, then the following conditions are equivalent:
		\begin{enumerate}
			\item $L^\Phi(K)$ is a Banach algebra under convolution;
			\item $L^\Phi(K)\subseteq L^1(K)$;
			\item $\lim_{x\rightarrow 0^+}\frac{\Phi(x)}{x}>0$ or $K$ is compact.
		\end{enumerate}
		
		The condition $L^\Phi(K) \subseteq L^1(K)$, where $K$ is a general hypergroup, has been studied by first author with his co-authors in \cite{Kumar}. This condition satisfies if  $(K,m)$ is a finite measure space or the  right derivative $\Phi'(0)>0$.
	\end{remark}

	In sequel of the paper, we assume that the aperiodic sequence is a subset of the center of hypergroup. We give a necessary condition for a weight to be densely hypercyclic.
	\begin{theorem}
		Let $\Phi$ be a Young function, $K$ be a locally compact hypergroup with a right Haar measure $m$, $w$ be a weight function on $K$, and $\eta:=(a_n)_{n\in\mathbb Z}$ be an aperiodic sequence in $Z(K)$.  If $w$ is  densely $(\eta,\Phi)$-hypercyclic, then for every compact subset $E\subseteq K$ with $m(E)>0$, there exist a sequence $(E_k)_{k=1}^\infty$ of subsets of $E$, and a strictly increasing sequence $(n_k)_{k=1}^\infty\subseteq\mathbb{N}$ such that $\lim_{k\rightarrow\infty}\|\chi_{E\setminus E_k}\|_\Phi=0$ and
		$$\lim_{k\rightarrow\infty}\|v^{-1}_{n_k}|_{E_k}\|_{\sup}=\lim_{k\rightarrow\infty}\|h_{n_k}|_{E_k}\|_{\sup}=0,$$
		where for all $n\in\mathbb{N}$ and $x\in K$,
		$v_n(x):=w(x)\cdot w(xa^-_{1})\cdot \ldots \cdot w(xa^-_{n}),$ and $h_n(x):=v_n(xa_{n}).$
	\end{theorem}
	\begin{proof}
		Suppose that $w$ is a  densely $(\eta,\Phi)$-hypercyclic weight on $K$. Let $E\subseteq K$ be compact and  $m(E)>0$.  By the hypothesis,  there exist  a sequence $(f_k)_{k=1}^\infty\subseteq L^\Phi(K)$ and a strictly increasing sequence $(n_k)_{k=1}^\infty\subseteq\mathbb{N}$ such that for each $k=1,2,\ldots$,
		$$\|f_k-\chi_E\|_\Phi<\frac{1}{4^{k}},\quad E\cap(Ea_{n_k})=\varnothing\quad\text{and}\quad \|\Lambda_{n_k}f_k-\chi_E\|_\Phi<\frac{1}{4^k}.$$
		
		Let $A_k:=\{x\in E: |f_k(x)-1|\geq\frac{1}{2^k}\}$. So, for each $x\in E\backslash A_k$, $|f_k(x)|>1-\frac{1}{2^k}$ and $\|\chi_{A_k}\|_\Phi<\frac{1}{2^{k}}$.
		Similarly, setting $B_k:=\{x\in K\backslash E:|f_k(x)|\geq\frac{1}{2^k}\}$ we have $\|\chi_{B_k}\|_\Phi<\frac{1}{2^{k}}$. Also, if we put $C_k:=\{x\in E: |v_{n_k}(x)f(xa_{-n_k})-1|\geq\frac{1}{2^k}\}$, then
		\begin{align*}
		\frac{1}{4^{k}}&>\|\Lambda_{n_k}f_k(x)-\chi_E\|_\Phi=\sup_{v\in\Omega_\Psi}\int_K |\Lambda_{n_k}f_k(x)-\chi_E(x)|\,|v(x)|\,dm(x)\\
		&\geq \sup_{v\in\Omega_\Psi}\int_{C_k} |w(x)\cdot w(xa_{-1})\cdot \ldots \cdot w(xa_{-n_k})\cdot f_k(xa_{-n_k})-1|\,|v(x)|\,dm(x)\\
		&=\sup_{v\in\Omega_\Psi}\int_{C_k} |v_{n_k}(x)f(xa_{-n_k})-1|\,|v(x)|\,dm(x)\geq \frac{1}{2^{k}}\,\|\chi_{C_k}\|_\Phi,
		\end{align*}
		and so, $\|\chi_{C_k}\|_\Phi<\frac{1}{2^{k}}$.
		Hence, for each $x\in E\backslash(C_k\cup B_k a_{n_k})$, we have
		\begin{align*}
		v_{n_k}(x)^{-1}<\frac{|f_k(xa_{-n_k})|}{1-\frac{1}{2^k}}<\frac{1}{2^k-1}.
		\end{align*}
		Put $D_k:=\{x\in E:|h_{n_k}(x)f_k(x)|\geq\frac{1}{2^k}\}$. Then,
		for each $x\in E\backslash(D_k\cup A_k)$, 
		$$h_{n_k}(x)<\frac{\frac{1}{2^k}}{|f_k(x)|}<\frac{1}{2^k-1}.$$
		Also,
		\begin{align*}
		\frac{1}{4^{k}}&>\|\Lambda_{n_k}f_k(x)-\chi_E\|_\Phi=\sup_{v\in\Omega_\Psi}\int_K |\Lambda_nf(x)-\chi_E(x)|\,|v(x)|\,dm(x)\\
		&=\sup_{v\in\Omega_\Psi}\int_K |w(x)\cdot w(xa_{-1})\cdot \ldots \cdot w(xa_{-n_k})\cdot f(xa_{-n_k})-\chi_E(x)|\,|v(x)|\, dm(x)\\
		&=\sup_{v\in\Omega_\Psi}\int_K |v_{n_k}(x)f_k(xa_{-n_k})-\chi_E(x)|\,|v(x)|\, dm(x)\\
		&=\sup_{v\in\Omega_\Psi}\int_K |v_{n_k}(xa_{n_k})f(x)-\chi_E(xa_{n_k})|\,|v(x)|\, dm(x)\\
		&=\sup_{v\in\Omega_\Psi}\int_K |v_{n_k}(xa_{n_k})f_k(x)-\chi_{Ea_{-n_k}}(x)|\,|v(x)|\, dm(x)\\
		&\geq\sup_{v\in\Omega_\Psi}\int_{D_k} |v_{n_k}(xa_{n_k})f_k(x)|\,|v(x)|\, dm(x)\geq\frac{1}{2^{k}}\|\chi_{D_k}\|_\Phi,
		\end{align*}
		and so $\|\chi_{D_k}\|_\Phi<\frac{1}{2^{k}}$. 
		Now, put $E_k:=E\backslash(A_k\cup(B_ka_{n_k})\cup C_k\cup D_k)$. Then, $\|\chi_{E\backslash E_k}\|_\Phi<\frac{4}{2^{k}}\rightarrow 0$, $\|h_{n_k}|_{E_k}\|_{\text{sup}}\rightarrow 0$ and  $\|v_{n_k}^{-1}|_{E_k}\|_{\text{sup}}\rightarrow 0$, as $k\rightarrow\infty$.
	\end{proof}
	\begin{theorem}\label{the35}
		Let $\Phi$ be a Young function with $\Phi\in\Delta_2$,  $K$ be a locally compact hypergroup, $w$ and $\frac{1}{w}$ be weights on $K$, and $\eta:=(a_n)_{n\in\mathbb Z}\subseteq Z(K)$ be an aperiodic sequence of elements in $K$.  Suppose that for every compact subset $E\subseteq K$ with $m(E)>0$, there exist a sequence $(E_k)_{k=1}^\infty$ of subsets of  $E$, and  a  sequence $(n_k)_{k=1}^\infty\subseteq\mathbb{N}$ with $n_1<n_2<\ldots$
		such that $\lim_{k\rightarrow\infty}\|\chi_{E\setminus E_k}\|_\Phi=0$ and
		$$\lim_{k\rightarrow\infty}\|v^{-1}_{n_k}|_{E_k}\|_{\sup}=\lim_{k\rightarrow\infty}\|h_{n_k}|_{E_k}\|_{\sup}=0,$$
		where for all $n\in\mathbb{N}$ and $x\in K$,
		$v_n(x):=w(x)\cdot w(xa^-_{1})\cdot \ldots \cdot w(xa^-_{n}),$ and $h_n(x):=v_n(xa_{n})$. Then, $w$ is densely $(\eta,\Phi)$-hypercyclic
	\end{theorem}
	\begin{proof}
		Suppose that $U$ and $V$ are non-empty open subsets of $L^\Phi(K)$. There are $f,g\in C_c(K)$ such that $f\in U$ and $g\in V$, since $\Phi\in\Delta_2$. 
		Put $E:=\text{supp}(f)\cup\text{supp}(g)$. Let the sequences $(E_k)$, $(h_{n_k})$ and $(v_{n_k})$ satisfy the hypothesis. For each $\epsilon>0$, there is a constant $N>0$ such that for each $k\geq N$, $E\cap E\ast\{a_{n_k}\}=E\cap E\ast\{a^-_{n_k}\}=\varnothing$, $\|h_{n_k}|_{E_k}\|_{\sup}\,\|f\|_\Phi<\epsilon$ and $\|v_{n_k}^{-1}|_{E_k}\|_{\sup}\,\|f\|_\Phi<\epsilon$. Hence, for each $k\geq N$ we have
		\begin{align*}
		\|\Lambda_{k}(f\chi_{E_{k}})\|_\Phi&=\sup_{v\in\Omega_\Psi}\int_K|v_{n_k}(x)|\,|f(x a_{-n_k})\chi_{E_{k}}(xa_{-n_k})|\,|v(x)|\,dm(x)\\
		&=\sup_{v\in\Omega_\Psi}\int_{{E_{k}}a_{n_k}}|v_{n_k}(x)|\,|f(x a_{-n_k})|\,|v(x)|\,dm(x)\\
		&=\sup_{v\in\Omega_\Psi}\int_{E_{k}}v_{n_k}(xa_{n_k})\,|f(x)|\,|v(x)|\,dm(x)\\
		&=\sup_{v\in\Omega_\Psi}\int_{E_{k}}h_{n_k}(x)\,|f(x)|\,|v(x)|\,dm(x)\\
		&\leq \|h_{n_k}|_{E_k}\|_{\text{sup}}\, \|f\|_\Phi<\epsilon.
		\end{align*}
		For each $n\in\mathbb{N}$ and $x\in K$, we define $S_n f(x):=v_n^{-1}(xa_n)f(xa_{n})$. So, 
		$$S_n(\Lambda_n f)(x)=v_n(xa_n)^{-1}(\Lambda_nf)(xa_n)=v_n(xa_n)^{-1}v_n(xa_n)f(xa_na_{-n})=f(x),$$ 
		and similarly, $\Lambda_n(S_nf)=f$. For each $k\geq N$ we have
		\begin{align*}
		\|S_{n_k}(f\chi_{E_{k}})\|_\Phi&=\sup_{v\in\Omega_\Psi}\int_K v_{n_k}^{-1}(xa_{n_k})\,|f(xa_{n_k})|\,\chi_{E_{k}}(xa_{n_k})\,|v(x)|\,dm(x)\\
		&=\sup_{v\in\Omega_\Psi}\int_{E_{k}}v_{n_k}^{-1}(x)\,|f(x)|\,|v(x)|\, dm(x)\\
		&\leq \|v_{n_k}^{-1}|_{E_k}\|_{\sup}\,\|f\|_\Phi<\epsilon.
		\end{align*}
		If for each $k\in\mathbb{N}$,
		$$v_k:=f\chi_{{E_k}}+S_{n_k}(g\chi_{E_{k}}),$$
		then $v_k\in L^\Phi(K)$, and
		\begin{align*}
		\|v_k-f\|_\Phi&=\|f\chi_{E\backslash E_k}-S_{n_k}(g\chi_{E_{k}})\|_\Phi\\
		&\leq \|f\chi_{E\backslash E_k}\|_\Phi+\|S_{n_k}(g\chi_{E_{k}})\|_\Phi\\
		&\leq \|f\|_{\text{sup}}\,\|\chi_{E\backslash E_k}\|_\Phi+\|S_{n_k}(g\chi_{E_{k}})\|_\Phi\rightarrow 0,
		\end{align*}
		as $k\rightarrow\infty$. Also,
		\begin{align*}
		\|\Lambda_{n_k}v_k-g\|_\Phi&=\|\Lambda_{n_k}(f\chi_{ E_k})+g\chi_{E_{k}}-g\|_\Phi\\
		&\leq \|\Lambda_{n_k}(f\chi_{ E_k})\|_\Phi+\|g\|_{\text{sup}}\,\|\chi_{E\backslash E_{k}}\|_\Phi\rightarrow 0,
		\end{align*}
		as $k\rightarrow \infty$. Then, $\Lambda_{n_k}(U)\cap V\neq\varnothing$, which implies that $w$ is densely $(\eta,\Phi)$-hypercyclic by
		\cite[Theorem 1.57]{gro}.
	\end{proof}
	The following form of the hypercyclic criterion is given in \cite{Bes} which is derived from original criterion obtained by Kitai \cite{Kitai}.
	\begin{lemma} \label{Kitai} Let $\mathcal X$ be a Fr$\acute{\text{e}}$chet space and  $T: \mathcal{X} \rightarrow \mathcal{X}$ be a bounded linear operator. Then, $T$ is  hypercyclic if it satisfies the following criteria: 
		\begin{itemize}
			\item[(i)] $(T^n)$ admits a subsequences $(T^{n_k})$ converging to zero pointwise on a dense subset of $\mathcal{X},$
			\item[(ii)] there exists a dense subset $Y$ of $\mathcal{X}$ and a sequence of maps $S^{n_k}:Y \rightarrow \mathcal{X}$ such that $(S^{n_k})$ tends to zero pointwise on $Y$ and $(T^{n_k}S^{n_k})$ tends to the identity pointwise on $Y.$ 
		\end{itemize}   
	\end{lemma}
	
	The Kitai's hypercyclic criteria as above is not necessary for the hypercyclicity of $T$ \cite{Bayart}. But it is equivalent to $T$ being  {\it hereditary hypercyclicity}, that is,  there is an increasing sequence $(n_k) \subset \mathbb{N}$ such that every subsequence $(T^{m_k})$ of $(T^{n_k})$ admits an element $x \in \mathcal X$ such that the set $\{T^{m_k}x\}_{k=1}^\infty$ is dense in $\mathcal X$.
	Our next theorem gives a necessary and sufficient condition for a weighted translation operator to be hereditary hypercyclic.
	It is clear from Lemma \ref{1.1} that if $w \equiv 1$ then $\|T_{1,z}\|=1$ and hence $T_{1,z}$ can not be hypercyclic. 
	In general, for $x \in K,$ $(fg)^x \neq f^xg^x$. But it can be observed that if $x \in Z(K)$ then we have $(fg)^x=f^xg^x$ \cite[Lemma 2.6]{che1}. 
	We use Kitai's criterion to give an equivalent condition for hereditary hypercyclicity of $T_{z,w}$ while $z\in Z(K)$ is an aperiodic element.

	\begin{theorem} Let $z \in Z(K)$ be an aperiodic element, $w$ be a weight on $K$, and $\Phi$ be a strictly increasing Young function with $\Phi\in\Delta_2$. Then, the followings are equivalent: 
		\begin{itemize}
			\item[(i)] $T_{z,w}$ is a hereditary hypercyclic operator on $L^\Phi(K)$.
			\item[(ii)] For each compact set $E \subset K$ with $m(E)>0$ there exists a sequence  $(E_k)$ of Borel subsets of $E$ such that $m(E)= \lim_{k \rightarrow \infty} m(E_k)$ and both sequences  
			$$w_{n}:= \prod_{j=1}^n w*\delta_{z^-}^j \,\,\,\,\, \text{and}\,\,\,\,\,\, \tilde{w}_n:= \left( \prod_{j=0}^{n-1} w*\delta_{z}^j \right)^{-1}$$ have subsequences $\{w_{n_k}\}$ and $\{\tilde{w}_{n_k}\}$ respectively such that $$\lim_{k \rightarrow \infty} \|w_{n_k}|_{E_k}\|_\infty= \lim_{k \rightarrow \infty}\|\tilde{w}_{n_k}|_{E_k}\|_\infty=0.$$
		\end{itemize}
	\end{theorem}
	\begin{proof} {\bf (i) $\Rightarrow$ (ii)}.
		Let the operator $T_{z,w}$ be hypercyclic on $L^\Phi(K)$, and $E \subset K$ be a compact set. Since $z$ is an aperiodic element it follows from Lemma \ref{t3.4}  that there exists a natural number $N \in \mathbb{N}$ such that $E \cap(E*(z^-)^n) = \varnothing.$ For a given $\epsilon>0$, we get $f \in L^\Phi(K)$ and $n_0 \in \mathbb{N}, n_0 \geq N$ such that $$N_\Phi(f-\chi_E) < \delta^2\,\,\,\,\,\,\text{and}\,\,\,\,\, N_\Phi(T_{z,w}^{n_0}f-\chi_E) < \delta^2,$$ where $\delta$ is choosen such that $0< \delta< \frac{\epsilon}{1+\epsilon}.$
		
		Put $A_{\delta}:= \{x \in E : |f(x)-1| \geq \delta\}.$ Then \begin{eqnarray*}
			\delta^2 > N_{\Phi}(f- \chi_E) \geq N_{\Phi}(\chi_E (f-1))& \geq & N_\Phi(\chi_{A_{\delta}}(f-1))\\ &\geq& N_\Phi(\chi_{A_{\delta}} \delta)\\ &=& \frac{\delta}{\Phi^{-1} \left( \frac{1}{m(A_{\delta})} \right)},
		\end{eqnarray*}
		which gives $m(A_{\delta})= \frac{1}{\Phi \left( \frac{1}{\delta} \right)}.$ 
		
		By a similar calculation, for $B_{\delta}:= \{x \in K \backslash E : |f(x)| \geq \delta \},$ we get $m(B_{\delta})< \frac{1}{\Phi \left( \frac{1}{\delta} \right)}.$
		
		Now, If we put $C_{n_0, \delta}:= \{x \in E: |\tilde{w}_{n_0}(x)^{-1} f(x*(z^-)^{n_0})-1| \geq \delta\}$, then we have 
		\begin{eqnarray*}
			\delta^2 &>& N_\Phi (T_{z,w}^{n_0}f-\chi_E) \geq  N_\Phi(\chi_{C_{n_0, \delta}}(T_{z,w}^{n_0}f- \chi_E)) \\&=& \inf \left\lbrace  k>0 : \int_{C_{n_0, \delta}} \Phi \left( \frac{1}{k} |\tilde{w}_{n_0}(x)^{-1} f(x*(z^-)^{n_0})-\chi_E(x)| \right) \, dm(x) \leq 1 \right\rbrace \\ &  \geq & N_\Phi (\delta \chi_{C_{n_0, \delta}})= \frac{\delta}{\Phi^{-1} \left( \frac{1}{m (C_{n_0, \delta})} \right)},
		\end{eqnarray*}
		which yields that $m(C_{n_0, \delta}) < \frac{1}{\Phi \left(\frac{1}{\delta} \right)}$.  Using $E \cap(E*(z^-)^{n_0}) = \varnothing$ we have 
		$$\tilde{w}_{n_0}(x)< \frac{|f(x*(z^-)^{n_0})|}{1-\delta}<\frac{\delta}{1-\delta}<\epsilon,\,\,\,\,\,\, x \in E \backslash (C_{n_0,\delta} \cup B_{\delta}z^{n_0}).$$
		
		Next, set $D_{n_0,\delta}:= \{x \in E: |w_{n_0}(x) f(x)| \geq \delta\}.$ Then, by Lemma \ref{1.1} and right invariance of right Haar measure $m,$ we get 
		\begin{eqnarray*}
			\delta^2 & > & N_\Phi(T_{z,w}^{n_0}f-\chi_E) \\ &=& \inf \left\lbrace k>0: \int_K \Phi\left( \frac{1}{k} |\tilde{w}_{n_0}(x)^{-1}f(x*(z^-)^{n_0})- \chi_E(x)| \right) \, dm(t) \leq 1  \right\rbrace  \\ &=& \inf \left\lbrace k>0: \int_K \Phi \left( \frac{1}{k} |w_{n_0}(x)f(x)- \chi_E(x*z^{n_0})| \right)\, dm(x) \leq 1 \right\rbrace \\ & \geq & \inf \left\lbrace k>0: \int_{D_{n_0, \delta}} \Phi \left(\frac{1}{k}| w_{n_0}(x)f(x)-\chi_E(x*z^{n_0})| \right)\, dm(x) \right\rbrace \\ &=& \inf \left\lbrace k>0: \int_{D_{n_0, \delta}} \Phi \left( \frac{1}{k} |w_{n_0}(x)f(x)| \, dm(t)| \right)\, dm(x) \leq 1\right\rbrace \\ &=& N_\Phi(\chi_{D_{n_0, \delta}} wf) \geq \delta N_\Phi(\chi_{D_{n_0, \delta}})= \frac{\delta}{\Phi^{-1} \left( \frac{1}{m(D_{n_0,\delta})} \right)},
		\end{eqnarray*}
		which implies that $m(D_{n_0, \delta}) < \frac{1}{\Phi \left( \frac{1}{\delta}\right)}.$ Now, we have $$w_{n_0}(x)< \frac{\delta}{|f(x)|}< \frac{\delta}{1-\delta}<\epsilon, \,\,\,\,\,\,x \in E \backslash (D_{n_0, \delta} \cup A_{\delta}).$$
		Finally, let $E_{n_0, \delta}:= E \backslash (A_{\delta} \cup B_{\delta} \cup C_{n_0, \delta} \cup D_{n_0, \delta}).$ Then,  it is clear that $m(E \backslash E_{n_0, \delta})< \frac{4}{\Phi(\frac{1}{\delta})}, \,\|w_{n_0}|_{E_{n_0, \delta}}\|_\infty< \epsilon$ and $\|\tilde{w}_{n_0}|_{E_{n_0, \delta}}\|_\infty < \epsilon.$
		This completes the proof.
		
		{\bf (ii) $\Rightarrow$ (i)}. We will prove that $T_{z,w}$ satisfies   Kitai's hypercyclicity criterion in Lemma \ref{Kitai}. Since $\Phi_2$ is $\Delta_2$-regular, $C_c(K)$ is dense in $L^\Phi(K)$. For each $m\in \mathbb N$ we define  the map $S_{z,w}^{m}:C_c(K) \rightarrow L^\Phi(K)$  by $$S_{z,w}^{m}(f)= \tilde{w}_{m} f*\delta_{z^-},\quad (f \in C_c(K)).$$
		
		Then, $T_{z,w}^{m}(S_{z,w}^{m}f)=f$ for all $f\in C_c(K)$.
		Fix a function $f\in C_c(K)$ and $\epsilon>0$.
		Let $(E_k)$, $(w_{n_k})$ and $(\tilde{w}_{n_k})$ be as in statement of the theorem satisfying condition (ii) related to the compact set $E:=\spt(f)$.   Now, we will show that $N_\Phi(T_{z,w}^{n_k}f) \rightarrow 0$ as $k \rightarrow \infty.$ Let $\{w_{n_k}\}$ be bounded by $M$ on the compact $\spt(f).$ For a given $\epsilon>0,$ by Egoroff 's Theorem there is  a Borel subset $E$ of $\spt(f)$ such that $m(\spt(f) \backslash E)< \frac{\epsilon}{M N_\Phi(f)}.$ Since $w_{n_k} \rightarrow 0$ uniformly on $E,$ there exists $N \in \mathbb{N}$ such that $w_{n_k}< \frac{\epsilon}{N_\Phi(f)}$ on $E.$ 
		Now, for $n_k>N$ we get
		\begin{align*}
		&N_\Phi(T_{z,w}^{n_k}f) = N_\Phi(T_{z,w}^{n_k}f \chi_{\spt(f)})\\
		&= \inf \left\lbrace k>0: \int_{\spt(f)z^{n_k}} \Phi ( \frac{1}{k} |w(x) w(x*z^{-1}) \ldots w(x*(z^-)^{n_k-1}) f(x*(z^-)^{n_k})| )\, dm(x) \leq 1\right\rbrace \\ &= \inf \left\lbrace k>0: \int_{\spt(f)} \Phi (\frac{1}{k} |w(x*z^{n_k})w(x*z^{n_k-1}) \ldots w(x*z)f(x)|) \, dm(x) \leq 1\right\rbrace \\ &= N_\Phi(w_{n_k}f \chi_E)+N_\Phi(w_{n_k} f \chi_{\spt(f) \backslash E}) \\&= \frac{\epsilon}{N_\Phi(f)} N_\Phi(f)+ \frac{2 \epsilon}{M N_\Phi(f)} M N_\Phi(f) =3 \epsilon.
		\end{align*}
		
		Using a similar argument one can see that $N_\Phi(S_{z,w}^{n_k}f) \rightarrow 0$ as $k \rightarrow \infty.$ So $T_{z,w}^{n_k}$ is hereditary hypercyclic on $L^\Phi(K)$ by \cite[Theorem 2.3]{Bes} as $C_c(K)$ is dense in $L^\Phi(K).$ 
		
	\end{proof}
	\section*{Acknowledgment}
	Vishvesh Kumar is supported by FWO Odysseus 1 grant G.0H94.18N: Analysis and Partial Differential Equations of Prof. Michael Ruzhansky. 
	\bibliographystyle{amsplain}
	
\end{document}